\documentclass[12pt, reqno]{amsart}
\usepackage{amsmath, amsthm, amscd, amsfonts, amssymb, graphicx, color}
\usepackage[bookmarksnumbered, colorlinks, plainpages]{hyperref}
\hypersetup{colorlinks=true,linkcolor=red, anchorcolor=green, citecolor=cyan, urlcolor=red, filecolor=magenta, pdftoolbar=true}

\newtheorem{theorem}{Theorem}[section]
\newtheorem{lemma}[theorem]{Lemma}
\newtheorem{proposition}[theorem]{Proposition}
\newtheorem{corollary}[theorem]{Corollary}
\theoremstyle{definition}
\newtheorem{definition}[theorem]{Definition}

\theoremstyle{remark}
\newtheorem{remark}[theorem]{Remark}
\numberwithin{equation}{section}

\begin{document}
%------------------------
\address{$^{[1,3]}$ Department of Mathematics, Jadavpur University, Kolkata 700032, West Bengal, India.}
\email{\url{pintubhunia5206@gmail.com} ; \url{kalloldada@gmail.com}}

\address{$^{[2]}$ University of Sfax, Tunisia.}
\email{\url{kais.feki@hotmail.com}}



\thanks{Pintu Bhunia would like to thank UGC, Govt. of India for the financial support in the form of SRF. Prof. Kallol Paul would like to thank RUSA 2.0, Jadavpur University for the partial support.}

\date{\today}
%\date{June 24, 2019}
\author[Pintu Bhunia, Kais Feki and Kallol Paul] {Pintu Bhunia$^{1}$, Kais Feki$^{2}$ and Kallol Paul$^{3}$ }
\title[$A$-Numerical radius orthogonality and parallelism ]{$A$-Numerical radius orthogonality and parallelism of semi-Hilbertian space operators and their applications}

\maketitle
%------------------------
%----------\thispagestyle{empty}

\begin{abstract}
In this paper, we aim to introduce and characterize the concept of numerical radius orthogonality of operators on a complex Hilbert space $\mathcal{H}$ which are bounded with respect to the semi-norm induced by a positive operator $A$ on $\mathcal{H}$. Moreover, a characterization of the $A$-numerical radius parallelism for $A$-rank one operators is proved. As applications of the obtained results, we obtain some $\mathbb{A}$-numerical radius inequalities of operator matrices where $\mathbb{A}$ is the operator diagonal matrix with diagonal entries are positive operator $A$. Some other related results are also investigated.
\end{abstract}

\section{Introduction and Preliminaries}\label{s1}
\noindent
Throughout this paper, $\mathcal{B}(\mathcal{H})$ denote the $C^{\ast}$-algebra of all bounded linear operators acting on a complex Hilbert space $\mathcal{H}$ with an inner product $\langle \cdot\mid \cdot \rangle$ and the corresponding norm $\|\cdot\| $. The symbol $I$ stands for the identity operator
on $\mathcal{H}$.
In all that follows, by an operator we mean a bounded linear operator. The range of every operator $T$ is denoted by $\mathcal{R}(T)$, its null space by $\mathcal{N}(T)$ and $T^*$ is the adjoint of $T$. Let $\mathcal{B}(\mathcal{H})^+$ be the cone of positive (semi-definite) operators, i.e.
$\mathcal{B}(\mathcal{H})^+=\left\{A\in \mathcal{B}(\mathcal{H})\,:\,\langle Ax\mid x\rangle\geq 0,\;\forall\;x\in \mathcal{H}\;\right\}$.
Every $A\in \mathcal{B}(\mathcal{H})^+$ defines the following positive semi-definite sesquilinear form:
$$\langle\cdot\mid\cdot\rangle_{A}:\mathcal{H}\times \mathcal{H}\longrightarrow\mathbb{C},\;(x,y)\longmapsto\langle x\mid y\rangle_{A} =\langle Ax\mid y\rangle.$$
Clearly, the induced semi-norm is given by $\|x\|_A=\langle x\mid x\rangle_A^{1/2}$, for every $x\in \mathcal{H}$. This makes $\mathcal{H}$ into a semi-Hilbertian space. One can verify that $\|\cdot\|_A$ is a norm on $\mathcal{H}$ if and only if $A$ is injective, and that $(\mathcal{H},\|\cdot\|_A)$ is complete if and only if $\mathcal{R}(A)$ is closed.
\begin{definition} (\cite{acg1})
Let $A\in\mathcal{B}(\mathcal{H})^+$ and $T \in \mathcal{B}(\mathcal{H})$. An operator $S\in\mathcal{B}(\mathcal{H})$ is called an $A$-adjoint of $T$ if for every $x,y\in \mathcal{H}$, the identity $\langle Tx\mid y\rangle_A=\langle x\mid Sy\rangle_A$ holds. That is $S$ is solution in $\mathcal{B}(\mathcal{H})$ of the equation $AX=T^*A$.
\end{definition}
\noindent The existence of an $A$-adjoint operator is not guaranteed. The set of all operators which admit $A$-adjoints is denoted by $\mathcal{B}_{A}(\mathcal{H})$. By Douglas Theorem \cite{doug}, we  have
\begin{align}\label{badeh}
\mathcal{B}_{A}(\mathcal{H})
& = \left\{T\in \mathcal{B}(\mathcal{H})\,:\;\mathcal{R}(T^{*}A)\subseteq \mathcal{R}(A)\right\}\nonumber\\
 &=\left\{T \in \mathcal{B}(\mathcal{H})\,:\;\exists \,\lambda > 0\,\textit{such that}\;\|ATx\| \leq \lambda \|Ax\|,\;\forall\,x\in \mathcal{H}  \right\}.
\end{align}
If $T\in \mathcal{B}_A(\mathcal{H})$, the reduced solution of the equation
$AX=T^*A$ is a distinguished $A$-adjoint operator of $T$, which is denoted by $T^{\sharp_A}$. Note that, $T^{\sharp_A}=A^\dag T^*A$ in which $A^\dag$ is the Moore-Penrose inverse of $A$. For more results concerning $T^{\sharp_A}$, see \cite{acg1,acg2}.
Again, by applying Douglas theorem, it can observed that
\begin{equation}\label{abbbbbbbb}
\mathcal{B}_{A^{1/2}}(\mathcal{H})=\left\{T \in \mathcal{B}(\mathcal{H})\,:\;\exists \,\lambda > 0\,\textit{such that}\;\|Tx\|_{A} \leq \lambda \|x\|_{A},\;\forall\,x\in \mathcal{H}  \right\}.
\end{equation}
Operators in $\mathcal{B}_{A^{1/2}}(\mathcal{H})$ are called $A$-bounded. Further, $\langle\cdot\mid\cdot\rangle_{A}$ induces the following semi-norm on $\mathcal{B}_{A^{1/2}}(\mathcal{H})$:
\begin{equation}\label{aseminorm}
\|T\|_A=\sup_{\substack{x\in \overline{\mathcal{R}(A)},\\ x\not=0}}\frac{\|Tx\|_A}{\|x\|_A}=\sup\left\{\|Tx\|_{A}\,;\;x\in \mathcal{H},\,\|x\|_{A}= 1\right\}<+\infty.
\end{equation}
For the rest of this paper, $A$ denotes a nonzero operator in $\mathcal{B}(\mathcal{H})^+$ and $P_A$ will be denoted to be the projection onto $\overline{\mathcal{R}(A)}$. Moreover, it is important to point out the following facts. The semi-inner product $\langle\cdot\mid\cdot\rangle_A$ induces an inner product on the quotient space $\mathcal{H}/\mathcal{N}(A)$ defined as
$$[\overline{x},\overline{y}]=\langle Ax\mid y\rangle,$$
for all $\overline{x},\overline{y}\in \mathcal{H}/\mathcal{N}(A)$. Notice that $(\mathcal{H}/\mathcal{N}(A),[\cdot,\cdot])$ is not complete unless $\mathcal{R}(A)$ is not closed. However, a canonical construction due to L. de Branges and J. Rovnyak in \cite{branrov} shows that the completion of $\mathcal{H}/\mathcal{N}(A)$  under the inner product $[\cdot,\cdot]$ is isometrically isomorphic to the Hilbert space $\mathcal{R}(A^{1/2})$
with the inner product
$$(A^{1/2}x,A^{1/2}y)=\langle P_Ax\mid P_Ay\rangle,\;\forall\, x,y \in \mathcal{H}.$$\\
\noindent In the sequel, the Hilbert space $\left(\mathcal{R}(A^{1/2}), (\cdot,\cdot)\right)$ will be denoted by $\mathbf{R}(A^{1/2})$ and we use the symbol
$\|\cdot\|_{\mathbf{R}(A^{1/2})}$ to represent the norm induced by the inner product $(\cdot,\cdot)$.  The interested reader is referred to \cite{acg3} for more information related to the Hilbert space $\mathbf{R}(A^{1/2})$. Notice that the fact $\mathcal{R}(A)\subset \mathcal{R}(A^{1/2}) $ implies that
\begin{align}\label{usefuleq001}
(Ax,Ay)
&=\langle x \mid y\rangle_A.
\end{align}
This leads to the following useful relation:
\begin{equation}\label{usefuleq01}
\|Ax\|_{\mathbf{R}(A^{1/2})}=\|x\|_A,\;\forall\,x\in \mathcal{H}.
\end{equation}
The following useful proposition is taken from \cite{acg3}.
\begin{proposition}\label{prop_arias}
Let $T\in \mathcal{B}(\mathcal{H})$. Then $T\in \mathcal{B}_{A^{1/2}}(\mathcal{H})$ if and only if there exists a unique $\widetilde{T}\in \mathcal{B}(\mathbf{R}(A^{1/2}))$ such that $Z_AT =\widetilde{T}Z_A$. Here, $Z_{A}: \mathcal{H} \rightarrow \mathbf{R}(A^{1/2})$ is
defined by $Z_{A}x = Ax$.
\end{proposition}
 Before we move on, it is important to state the following useful lemmas.
\begin{lemma}$($\cite{kais01}$)$\label{preprintkais}
If $T\in \mathcal{B}_{A^{1/2}}(\mathcal{H})$, then we have
\begin{itemize}
  \item [(1)] $\|T\|_A=\|\widetilde{T}\|_{\mathcal{B}(\mathbf{R}(A^{1/2}))}$.
  \item [(2)] $\omega_A(T)=\omega(\widetilde{T})$,
\end{itemize}
where $\omega(\widetilde{T})$ is the numerical radius of $\widetilde{T}$ and $\omega_A(T)$ is the $A$-numerical radius of $T$, defined as below.
\end{lemma}

\begin{lemma}\label{lau}$($\cite[Proposition 2.9]{majsecesuci}$)$
Let $T\in \mathcal{B}_A(\mathcal{H})$. Then,
$$\widetilde{T^{\sharp_A}}=\big(\widetilde{T}\big)^*\;\text{ and }\; \widetilde{({T^{\sharp_A}})^{\sharp_A}}=\widetilde{T}.$$
\end{lemma}

\noindent Now we note that $\mathcal{B}_{A}(\mathcal{H})$ and $\mathcal{B}_{A^{1/2}}(\mathcal{H})$ are two subalgebras of $\mathcal{B}(\mathcal{H})$ which are neither closed nor dense in $\mathcal{B}(\mathcal{H})$. Moreover, the following inclusions $\mathcal{B}_{A}(\mathcal{H})\subseteq\mathcal{B}_{A^{1/2}}(\mathcal{H})\subseteq
 \mathcal{B}(\mathcal{H})$ hold with equality if $A$ is injective and has closed range. For an account of results, we refer to \cite{acg1,acg2,kais01}.

It is useful to recall that an operator $T$ is called $A$-self-adjoint if $AT$ is self-adjoint (i.e. $AT=T^*A$) and it is called $A$-positive if $AT\geq0$.

\noindent Recently, several results covering some classes of operators on a Hilbert space $\mathcal{H}$ were extended when an additional semi-inner
product defined by $A\in \mathcal{B}(\mathcal{H})^+$ is considered. One may see \cite{BPN, BP, bakfeki01,bakfeki04,majsecesuci,zamani1}.

\noindent The generalization of the numerical range, known as $A$-numerical range (see \cite{bakfeki02}) is given by:
\[W_A(T) = \{\langle Tx\mid x\rangle_A: x\in \mathcal{H}, \|x\|_A=1\}.\]
  The $A$-numerical radius $\omega_A(T)$ and the $A$-Crawford number $m_A(T)$ of an operator $T$ are defined as:
\[ \omega_A(T)=\sup\{|\lambda|: \lambda \in W_A(T)\}, \]
\[m_A(T)=\inf\{|\lambda|: \lambda \in W_A(T)\}.\]
It is well-known that the $A$-numerical radius of an $A$-bounded operator $T$ is equivalent to $A$-operator semi-norm of $T$, (see \cite{zamani1}). More precisely, we have
\begin{equation}\label{equivalentsemi}
\tfrac{1}{2}\|T\|_A \leq \omega_A(T)\leq \|T\|_A.
\end{equation}
Zamani \cite{zamani1} studied $A$-numerical radius inequalities for semi-Hilbertian space operators. In \cite{BPN,BP}, we have also studied $\mathbb{A}$-numerical radius inequalities of $d\times d$ operator matrices where $\mathbb{A}$ is the $d\times d$ diagonal operator matrix whose diagonal entries are $A$. Notice that the study of numerical radius inequalities received considerable attention in the last decades (the reader is invited to consult for example \cite{BBP,BBP3} and the references therein). An operator $U\in  \mathcal{B}_A(\mathcal{H})$ is said to be $A$-unitary if $U^{\sharp_A} U=(U^{\sharp_A})^{\sharp_A} U^{\sharp_A}=P_A$. We mention that if $T\in \mathcal{B}_A(\mathcal{H})$ then $T^{\sharp_A}\in \mathcal{B}_A(\mathcal{H})$, $(T^{\sharp_A})^{\sharp_A}=P_ATP_A$.
Let $\mathbb{T}$ denote the unit cycle of the complex plane, i.e. $\mathbb{T}=\{\lambda\in \mathbb{C}\,:\;|\lambda|=1 \}$. Recently, new types of parallelism for $A$-bounded operators based on the $A$-numerical radius and the $A$-operator semi-norm was introduced in \cite{fekisidha2019}. More precisely, we have the following definition.
\begin{definition}(\cite{fekisidha2019})
Let $T,S\in \mathcal{B}_{A^{1/2}}(\mathcal{H})$.
\begin{itemize}
  \item [(1)]  We say that $T$ is $A$-norm-parallel to $S$, in short $T\parallel_A S$, if there exists $\lambda\in \mathbb{T}$ such that
$$\|T+\lambda S\|_A=\|T\|_A+\|S\|_A.$$
  \item [(2)] The operator $T$ is said to be $A$-numerical radius parallel to $S$ and we denote $T \parallel_{\omega_A} S$, if
\begin{align*}
\omega_A(T + \lambda S) = \omega_A(T)+\omega_A(S) \qquad \mbox{for some}\,\, \lambda\in\mathbb{T}.
\end{align*}
\end{itemize}
\end{definition}
\noindent Also, the following theorems are proved in \cite{fekisidha2019}.
\begin{theorem}\label{main1}
Let $T,S\in \mathcal{B}_{A^{1/2}}(\mathcal{H})$. Then, the following assertions are equivalent:
\begin{itemize}
  \item [(1)] $T\parallel_AS$.
  \item [(2)] There exists a sequence $(x_n)_n\subset\mathcal{H}$ such that $\|x_n\|_A=1$,
    \begin{equation*}
    \lim_{n\to +\infty}|\langle T x_n\mid Sx_n\rangle_A|=\|T\|_A\|S\|_A.
    \end{equation*}
\end{itemize}
\end{theorem}

\begin{theorem}\label{main2}
Let $T,S\in \mathcal{B}_{A^{1/2}}(\mathcal{H})$. Then the following conditions are equivalent:
\begin{itemize}
\item[(1)] $T \parallel_{\omega_A} S$.
\item[(2)] There exists a sequence of $A$-unit vectors $\{x_n\}$ in $\mathcal{H}$ such that
\begin{equation}\label{importanteqf}
\lim_{n\rightarrow+\infty} \big|\langle Tx_n\mid x_n\rangle_A\langle Sx_n\mid x_n\rangle_A\big| = \omega_A(T)\omega_A(S).
\end{equation}
\end{itemize}
\end{theorem}

\noindent Recently, Zamani introduced in \cite{Z.3} the notion of $A$-Birkhoff-James orthogonality of operators in semi-Hilbertian spaces as follows.
\begin{definition}(\cite{Z.3})\label{de.31}
An element $T \in \mathcal{B}_{A^{1/2}}(\mathcal{H})$ is called an $A$-Birkhoff-James orthogonal
to another element $S \in \mathcal{B}_{A^{1/2}}(\mathcal{H})$, denoted by $T\perp^B_A S$, if
\begin{align*}
{\|T + \gamma S\|}_A\geq {\|T\|}_A \quad \mbox{for all} \,\, \gamma \in \mathbb{C}.
\end{align*}
\end{definition}

The paper is organized as follows: In the next section, we introduce and give a characterization of $A$-numerical radius orthogonality for $A$-bounded
operators. In particular, for $T, S \in \mathcal{B}_{A^{1/2}}(\mathcal{H})$, we show that $T$ is $A$-numerical radius orthogonal to $S$
if and only if for each $\beta\in [0,2\pi)$, there exists a sequence of $A$-unit vectors $\{x_k\}$ in $\mathcal{H}$ such that
$$
  \lim_{k\to +\infty}|\langle Tx_k\mid x_k\rangle_A|=\omega_A(T) \text{ and }\lim_{k\to +\infty}\Re e\left( e^{i\beta}\langle x_k\mid Tx_k\rangle_A \langle Sx_k\mid x_k\rangle_A\right)\geq 0.
$$
Furthermore, inspiring by the rank one operators in Hilbert spaces, we introduce the class of $A$-rank one operators in semi-Hilbert spaces. In addition, a characterization of the $A$-numerical radius parallelism of $A$-rank one operators is established. Our results cover and extend the works in \cite{malpaulsen,mehamzamani}.
In the last section, we give some inequalities for $\mathbb{A}$-numerical radius of semi-Hilbertian space operators which are as an application of $\mathbb{A}$-numerical radius orthogonality and parallelism. The obtained results generalize and improve on the existing inequalities.

\section{$A$-numerical radius orthogonality and parallelism }
\noindent In this section, we introduce and completely characterize the concept of orthogonality of $A$-bounded operators with respect to the $A$-numerical radius $\omega(\cdot)$. Also we give a characterization of $A$-numerical radius parallelism for $A$-rank one operators. First, let us introduce the notion of $A$-numerical radius orthogonality of operators in semi-Hilbertian spaces.
\begin{definition}
An element $T \in \mathcal{B}_{A^{1/2}}(\mathcal{H})$ is called an $A$-numerical radius orthogonal to another element $S \in \mathcal{B}_{A^{1/2}}(\mathcal{H})$, denoted by $T\perp_{\omega_A} S$, if
$$
\omega_A(T + \gamma S)\geq \omega_A(T) \quad \mbox{for all} \,\, \gamma \in \mathbb{C}.
$$
\end{definition}

\noindent In the following proposition we state some basic properties of $A$-numerical radius orthogonality. The proof follows immediately from the definition of $A$-numerical radius orthogonality of operators and hence it is omitted.
\begin{proposition}
Let $T,S\in \mathcal{B}_A(\mathcal{H})$. Then the following properties are equivalent.
\begin{itemize}
\item [(i)] $T\perp_{\omega_A} S$.
\item [(ii)] $T^{\sharp_A}\perp_{\omega_A} S^{\sharp_A}.$
\item [(iii)] $\alpha T\perp_{\omega_A}\beta S$ for all $\alpha,\beta \in \mathbb{C}\setminus \{0\}$.
\end{itemize}
\end{proposition}

\noindent In the next proposition, we give some connections between $A$-numerical radius orthogonality and $A$-Birkhoff-James orthogonality of operators. Recall
from \cite{kais01} that an operator $T\in\mathcal{B}_{A^{1/2}}(\mathcal{H})$ is said to be $A$-normaloid if $r_A(T)=\|T\|_A$, where
$$r_A(T)=\displaystyle\lim_{n\to+\infty}\|T^n\|_A^{\frac{1}{n}}.$$
Moreover, it was shown in \cite{kais01} that an operator $T\in\mathcal{B}_{A^{1/2}}(\mathcal{H})$ is $A$-normaloid if and only if $\omega_A(T)=\|T\|_A$.

Now we give the following proposition.
\begin{proposition}
Let $T,S\in \mathcal{B}_{A^{1/2}}(\mathcal{H})$. Then the following conditions hold:
\begin{itemize}
  \item [(1)] If $T$ is an $A$-normaloid operator, then $T\perp_{\omega_A} S\Rightarrow T\perp^B_A S$.
  \item [(2)] If $AT^2=0$, then $T\perp^B_A S\Rightarrow T\perp_{\omega_A} S$.
\end{itemize}
\end{proposition}
\begin{proof}
\noindent (1)\;Notice first that since $T$ is an $A$-normaloid operator, then $\omega_A(T)=\|T\|_A$. Now, assume that $T\perp_{\omega_A} S$. This implies that $\omega_A(T+\lambda S)\geq \omega_A(T)$ for all $\lambda \in \mathbb{C}$. Hence, by taking into account \eqref{equivalentsemi}, we get
$$\|T+\lambda S\|_A\geq \omega_A(T+\lambda S)\geq \omega_A(T)=\|T\|_A,$$
for all $\lambda \in \mathbb{C}.$ Therefore, we deduce that $T\perp^B_A S$.
\par \vskip 0.1 cm \noindent (2)\;It was shown in \cite[Corollary 2.2.]{kais01} that if $AT^2=0,$ then $\omega_A(T)=\frac{1}{2}\|T\|_A$. Now, we assume that $T\perp^B_A S$. Then, $\|T+\lambda S\|_A\geq \|T\|_A$ for all $\lambda \in \mathbb{C}$. Moreover, by using \eqref{equivalentsemi}, we see that
$$\omega_A(T+\lambda S)\geq \tfrac{1}{2}\|T+\lambda S\|_A\geq \tfrac{1}{2}\|T\|_A=\omega_A(T),$$
for all $\lambda \in \mathbb{C}$. Thus, $T\perp_{\omega_A} S$ as required.
\end{proof}

\begin{remark}
In general, as it was point out in \cite{malpaulsen} that the above two notions of orthogonality are not equivalent.
\end{remark}

\noindent In the following theorem, we prove our first main result in this section, which  characterizes $A$-numerical radius orthogonality of $A$-bounded operators on complex Hilbert space.
\begin{theorem}\label{mainortho}
Let $T,S\in \mathcal{B}_{A^{1/2}}(\mathcal{H})$. Then, the following assertions are equivalent:
\begin{itemize}
  \item [(1)] $T\perp_{\omega_A} S$.
  \item [(2)] For each $\beta\in [0,2\pi)$, there exists a sequence of $A$-unit vectors $\{x_k\}$ in $\mathcal{H}$ such that
  \begin{equation}\label{perortho}
  \lim_{k\to +\infty}|\langle Tx_k\mid x_k\rangle_A|=\omega_A(T) \text{ and }\lim_{k\to +\infty}\Re e\left( e^{i\beta}\langle x_k\mid Tx_k\rangle_A \langle Sx_k\mid x_k\rangle_A\right)\geq 0.
  \end{equation}
\end{itemize}
\end{theorem}
\begin{proof}
$(1)\Longrightarrow(2):$ Assume that $T\perp_{\omega_A} S$. Then, $\omega_A(T+\lambda S)\geq \omega_A(T),$ for every $\lambda \in \mathbb{C}$. Let $\beta \in [0,2\pi)$. It follows from the definition of $\omega_A(\cdot)$ that, for every $n\in \mathbb{N}^*$, there exists a sequence of $A$-unit vectors $\{z_n\}$ in $\mathcal{H}$ such that
\begin{equation}\label{copiepaste}
|\langle Tz_n+\tfrac{e^{i\beta}}{n} Sz_n\mid z_n\rangle_A|>\omega_A(T)-\tfrac{1}{n^2}.
\end{equation}
Therefore, for all $n\in \mathbb{N}^*$ we have
$$(\omega_A(T)-\tfrac{1}{n^2})^2< |\langle Tz_n+\tfrac{e^{i\beta}}{n} Sz_n\mid z_n\rangle|^2.$$
This implies that
\begin{align*}
\omega^2_A(T)-\tfrac{2}{n^2}\omega_A(T)+\tfrac{1}{n^4}
 &< |\langle Tz_n\mid z_n\rangle_A|^2+\tfrac{1}{n^2} |\langle Sz_n\mid z_n\rangle_A|^2  \\
		& \;\;+ \tfrac{2}{n}\Re e\left( e^{i\beta}\langle z_n\mid Tz_n\rangle_A \langle Sz_n\mid z_n\rangle_A\right),
\end{align*}
which in turn yields that
\begin{align*}
\tfrac{n}{2} \left[\omega^2_A(T)-|\langle Tz_n\mid z_n\rangle_A|^2\right]
 &< \tfrac{1}{n}\omega_A(T)-\tfrac{1}{2n^3} +\tfrac{1}{2n} |\langle Sz_n\mid z_n\rangle_A|^2\\
		&\;\;+ \Re e\left( e^{i\beta}\langle z_n\mid Tz_n\rangle_A \langle Sz_n\mid z_n\rangle_A\right).
\end{align*}
Hence, we infer that
\begin{equation}\label{ifnecess}
\tfrac{1}{n}\omega_A(T)-\tfrac{1}{2n^3} +\tfrac{1}{2n}\|S\|_A^2 +\Re e\left( e^{i\beta}\langle z_n\mid Tz_n\rangle_A \langle Sz_n\mid z_n\rangle_A\right)>0,
\end{equation}
for all $n\in \mathbb{N}^*$. Moreover, since $T,S\in \mathcal{B}_{A^{1/2}}(\mathcal{H})$ and $\|z_n\|_A=1$, then by the Cauchy-Schwarz inequality it can be seen that $(\langle Tz_n\mid z_n\rangle_A)_n$ and $(\langle Sz_n\mid z_n\rangle_A)_n$ are bounded sequences of complex numbers. So, there exists a subsequence $(z_{n_k})_k$ of $(z_n)_n$ such that
$$\displaystyle\lim_{k\to+\infty}\Re e\left( e^{i\beta}\langle z_{n_k}\mid Tz_{n_k}\rangle_A \langle Sz_{n_k}\mid z_{n_k}\rangle_A\right)\;\text{ exists}.$$
 Now, consider the sequence $\{x_k\}$ such that $x_k=z_{n_k}$ for all $k$. Clearly, $\|x_k\|_A=1$ for all $k$. Moreover, by using \eqref{ifnecess} we get
$$\lim_{k\to+\infty}\Re e\left( e^{i\beta}\langle x_k\mid Tx_k\rangle_A \langle Sx_k\mid x_k\rangle_A\right)\geq 0,$$
as desired. On the other hand, by using the Cauchy-Schwarz inequality and taking into consideration \eqref{copiepaste} together with the fact that $n_k\geq k$ for all $k$, we obtain
\begin{align*}
|\langle Tx_k\mid x_k\rangle_A|
&\geq|\langle Tx_k+\tfrac{e^{i\beta}}{k} Sx_k\mid x_k\rangle_A|-\tfrac{1}{k}|\langle Sx_k\mid x_k\rangle_A| \\
 &>\omega_A(T)-\tfrac{1}{k^2}-\tfrac{1}{k}\|S\|_A.
\end{align*}
So, by letting $k\to+\infty$, we obtain $\displaystyle\lim_{k\to+\infty} |\langle Tx_k\mid x_k\rangle_A|\geq \omega_A(T)$. This immediately gives
$$\omega_A(T)=\lim_{k\to+\infty} |\langle Tx_k\mid x_k\rangle_A|,$$
as required. Thus, the assertion $(2)$ is proved.

$(2)\Longrightarrow(1):$ Let $\lambda \in \mathbb{C}$. Then, there exists some $\beta \in [0,2\pi)$ such that $\lambda =|\lambda|e^{i\beta}$. So, by hypothesis, there exists a sequence of $A$-unit vectors $\{x_k\}$ in $\mathcal{H}$ such that \eqref{perortho} holds. Hence, we see that
\begin{align*}
\omega^2_A(T+\lambda S)
&\geq \lim_{k\to+\infty}|\langle Tx_k+\lambda Sx_k\mid x_k\rangle_A|^2\\
&= \lim_{k\to +\infty}[|\langle Tx_k\mid x_k\rangle_A|^2+2|\lambda|\,\Re e\left( e^{i\beta}\langle x_k\mid Tx_k\rangle_A \langle Sx_k\mid x_k\rangle_A\right)\\
&\;\;+|\lambda|^2 |\langle Sx_k\mid x_k\rangle_A|^2]	\\
&\geq\lim_{k\to+\infty}|\langle Tx_k\mid x_k\rangle_A|^2\\
&=\omega^2_A(T).
\end{align*}
Thus, $\omega_A(T+\lambda A)\geq \omega_A(T)$ for every $\lambda \in \mathbb{C}$. Therefore, $T\perp_{\omega_A}S$. Hence, the proof of the theorem is complete.
\end{proof}
We would like to emphasize the following remark.
\begin{remark}
For every $\beta\in [0,2\pi)$, the limit of $\Re e\left( e^{i\beta}\langle z_n\mid Tz_n\rangle_A \langle Sz_n\mid z_n\rangle_A\right)$ need not exist in general for some sequence $\{z_n\}$ of $A$-unit vectors even if $A=I$. Indeed, let $\mathcal{H}=\ell_{\mathbb{N}^*}^2(\mathbb{C})$ and $(e_n)_{n\in \mathbb{N}^*}$ be the canonical basis of $\ell_{\mathbb{N}^*}^2(\mathbb{C})$. Assume that $\theta = 0$ and $A =T=I$. Consider the following operator
$$S: \ell_{\mathbb{N}^*}^2(\mathbb{C}) \to \ell_{\mathbb{N}^*}^2(\mathbb{C}),\;x=(x_n)_{n\in \mathbb{N}^*}\mapsto Sx=(y_n)_{n\in \mathbb{N}^*}, $$
such that $y_n = (-1)^n x_n$ for all $n\in \mathbb{N}^*$. Clearly, $S\in \mathcal{B}(\ell_{\mathbb{N}^*}^2(\mathbb{C}))$. Moreover,
$$\Re e\left(\langle Se_n\mid e_n \rangle\right) = (-1)^n.$$
\end{remark}

\noindent Now we aim to give a characterization of the $A$-numerical radius parallelism for special type of semi-Hilbert space operators which will be called $A$-rank one operators. This new class of operators is defined as follows.
\begin{definition}
 Let $x,y\in \mathcal{H}$, the $A$-rank one operator is denoted by $x\otimes_A y$, where $x\otimes_A y$ is the following map:
\begin{align*}
x\otimes_A y\colon \mathcal{H} & \rightarrow\mathcal{H}\\
z&\mapsto (x\otimes_A y)(z)=\langle z\mid y\rangle_Ax.
\end{align*}
\end{definition}

\noindent In order to characterize of the $A$-numerical radius parallelism for $A$-rank one operators, we need the following two lemmas. Before that, it is useful to recall
for every $x, y \in \mathcal{H}$, the rank one operator $x\otimes y$ verifies
\begin{equation}\label{rankone}
\|x\otimes y\|=\|x\| \|y\|\;\;\text{ and }\;\;\omega(x\otimes y)=\tfrac{1}{2}\left(|\langle x\mid y\rangle|+\|x\| \|y\|\right).
\end{equation}

\begin{lemma}\label{l1}
Let $x,y\in \mathcal{H}$. Then, the following properties hold.
\begin{itemize}
  \item [(1)] $\|x\otimes_A y\|_A=\|x\|_A \|y\|_A$.
  \item [(2)] $\omega_A(x\otimes_A y)=\tfrac{1}{2}\left(\langle x\mid y\rangle_A+\|x\|_A\|y\|_A \right).$
\end{itemize}
\end{lemma}
\begin{proof}
\noindent (1)\;Let $z\in \mathcal{H}$. Then, by using the Cauchy-Schwarz inequality, we have
\begin{align*}
\|(x\otimes_A y)(z)\|_A
& =\|A^{1/2}(x\otimes_A y)(z)\|\\
& = \|\langle z\mid y\rangle_AA^{1/2}x\|\\
 &=|\langle A^{1/2}z\mid A^{1/2}y\rangle|\times \|x\|_A\\
  &\leq  \|x\|_A\|y\|_A\|z\|_A.
\end{align*}
Hence, $x\otimes_A y\in \mathcal{B}_{A^{1/2}}(\mathcal{H})$. So, by Proposition \ref{prop_arias} there exists a unique $\widetilde{x\otimes_A y}\in \mathcal{B}(\mathbf{R}(A^{1/2}))$ such that $Z_A(x\otimes_A y) =\widetilde{x\otimes_A y}Z_A$. On the other hand, by using \eqref{usefuleq001} we see that
\begin{align*}
Z_A(x\otimes_A y)(z)
&=\langle z\mid y\rangle_AAx \\
&=(Az,Ay)Ax \\
 &=(Ax\otimes Ay)(Az).
\end{align*}
for all $z\in \mathcal{H}$. Hence, $\widetilde{x\otimes_A y}=Ax\otimes Ay$, where $Ax\otimes Ay$ is defined as
\begin{align*}
Ax\otimes A y\colon \mathbf{R}(A^{1/2}) & \rightarrow\mathbf{R}(A^{1/2})\\
z&\mapsto (Ax\otimes Ay)(z)=(z, Ay)Ax.
\end{align*}
So, by using \eqref{rankone} we obtain
\begin{align*}
\|x\otimes_A y\|_A
&=\|\widetilde{x\otimes_A y}\|_{\mathcal{B}(\mathbf{R}(A^{1/2}))} \\
&=\|Ax\otimes Ay\|_{\mathcal{B}(\mathbf{R}(A^{1/2}))} \\
 &=\|Ax\|_{\mathbf{R}(A^{1/2})} \|Ay\|_{\mathbf{R}(A^{1/2})}\\
  &=\|x\|_A \|y\|_A.
\end{align*}
\par \vskip 0.1 cm \noindent (2)\;By applying Lemma \ref{preprintkais} together with \eqref{rankone} we get
\begin{align*}
\omega_A(x\otimes_A y)
& =\omega(\widetilde{x\otimes_A y}) \\
& =\omega(Ax\otimes Ay) \\
 &=\tfrac{1}{2}\left[(A x, Ay)+\|Ax\|_{\mathbf{R}(A^{1/2})}\|Ay\|_{\mathbf{R}(A^{1/2})} \right]\\
  &=\tfrac{1}{2}\left(\langle x\mid y\rangle_A+\|x\|_A\|y\|_A \right)(\text{ by } \eqref{usefuleq001}\text{ and } \eqref{usefuleq01}).
\end{align*}
\end{proof}

\begin{remark}
Very recently, as our work was in progress, the above lemma has been proved by Zamani in \cite{z2019}. Our proof here is different from his approach.
\end{remark}

\begin{lemma}\label{l2}
Let $x,y\in \mathcal{H}$. Then, the following properties are equivalent.
\begin{itemize}
  \item [(1)] $x\parallel_A y$ $($ i.e. $\|x+\lambda y\|_A=\|x\|_A+\|y\|_A$ for some $\lambda\in\mathbb{T}$ $)$.
  \item [(2)] $x\otimes_Ay \parallel_A I$.
\end{itemize}
\end{lemma}
\begin{proof}
By using \eqref{usefuleq01}, one can observe that $x\parallel_A y$ if and only if $Ax\parallel A y$ (that is $Ax$ and $Ay$ are parallel on $\mathbf{R}(A^{1/2})$). On the other hand, it is well-known that for a given $T,S\in \mathcal{B}_{A^{1/2}}(\mathcal{H})$ we have $T\parallel_AS$ if and only if $\widetilde{T}\parallel \widetilde{S}$ (see \cite[Lemma 3.1]{fekisidha2019}). So, $x\otimes_Ay \parallel_A I$ if and only if $\widetilde{x\otimes_Ay} \parallel \widetilde{I}$. Since, $\widetilde{x\otimes_Ay}=Ax \otimes Ay$ and $\widetilde{I}=I_{\mathbf{R}(A^{1/2})}$, then we get the desired equivalence by applying \cite[Corollary 2.23]{Z.M.2}.
\end{proof}

Now, we have in a position to prove our second main result in this section.
\begin{theorem}
Let $x,y\in \mathcal{H}$. Then the following conditions are equivalent:
\begin{itemize}
\item[(1)] $x\parallel_A y$.
\item[(2)] $x\otimes_A x \parallel_{\omega_A} y\otimes_A y$.
\end{itemize}
\end{theorem}
\begin{proof}
$(1)\Rightarrow(2):$ Suppose that $x\parallel_A y$. Then, in view of Lemma \ref{l2}, we have $x\otimes_A y\parallel_A I$. So, by Theorem \ref{main1} there exists a sequence of $A$-unit vectors $\{x_n\}$ in $\mathcal{H}$ such that
 \begin{align}\label{new}
\lim_{n\rightarrow+\infty} \big|\langle (x\otimes_A y)x_n\mid x_n\rangle_A\big| = \|x\otimes_A y\|_A.
\end{align}
On the other hand, by Lemma \ref{l1} we have
 \begin{align}\label{marj1}
\omega_A(x\otimes_A x)\omega_A(y\otimes_A y) =\|x\|_A^2\|y\|_A^2 = \|x\otimes_A y\|_A^2.
 \end{align}
Moreover, a short calculation shows that
\begin{align}\label{marj2}
\langle (x\otimes_A x)x_n\mid x_n\rangle_A\langle(y\otimes_A y)x_n\mid x_n\rangle_A
&=\big|\langle x\mid x_n\rangle_A\langle x_n\mid y\rangle_A\big|^2\nonumber\\
&= \big|\langle (x\otimes_A y)x_n\mid x_n\rangle_A\big|^2
\end{align}
So, by combining \eqref{new} together with \eqref{marj1} and \eqref{marj2}, we get
 \begin{align}\label{marjjdid}
\lim_{n\rightarrow+\infty}\Big|\langle (x\otimes_A x)x_n\mid x_n\rangle_A\langle(y\otimes_A y)x_n\mid x_n\rangle_A\Big|
=\omega_A(x\otimes_A x)\omega_A(y\otimes_A y).
\end{align}
Therefore, by Theorem \ref{main2} we deduce that $x\otimes_A x \parallel_{\omega_A} y\otimes_A y$ as required.

$(2)\Rightarrow(1):$ Assume that $x\otimes_A x\parallel_{\omega_A} y\otimes_A y$. It follows from Theorem \ref{main2} that there exists a sequence of $A$-unit vectors $\{x_n\}$ in $\mathcal{H}$ such that \eqref{marjjdid} holds. So, by making the same computations as above and applying Theorem \ref{main1} together with
Lemma \ref{l2} we can easily show that $x\parallel_A y$. This completes the proof of the theorem.
\end{proof}

In the following theorem, we study the connection between $A$-numerical radius parallelism and $A$-semi-norm parallelism.
\begin{theorem}
Let $T_{12},T_{21}\in \mathcal{B}_A(\mathcal{H})$ be $A$-normaloid operators. If $T_{12} \parallel_{\omega_A} T_{21}$, then $T_{12} \parallel_{A} T_{21}$.
\end{theorem}

\begin{proof}
Since $T_{12} \|_{\omega_A} T_{21}$, then there exists some $\lambda \in \mathbb{T}$ such that $\omega_{A}(T_{12}+ \lambda T_{21})=\omega_A(T_{12})+\omega_A(T_{21})$. In addition, since $T_{12}$ and $T_{21}$ are $A$-normaloid, then we have $\omega_A(T_{12})=\|T_{12}\|_A$ and $\omega_A(T_{21})=\|T_{21}\|_A$. Therefore, we obtain
\begin{align*}
\|T_{12}\|_A+\|T_{21}\|_A
&= \omega_A(T_{12})+\omega_A(T_{21})\\
&= \omega_{A}(T_{12}+ \lambda T_{21})\\
&\leq  \|T_{12}+ \lambda T_{21}\|_A\\
&\leq \|T_{12}\|_A+\|T_{21}\|_A.
\end{align*}
This implies that $\|T_{12}+ \lambda T_{21}\|_A=\|T_{12}\|_A+\|T_{21}\|_A$ for some $\lambda \in \mathbb{T}$ and so $T_{12} \|_{A} T_{21}$ as required.
\end{proof}

\section{Application: some $A$-numerical radius inequalities}

\noindent In this section, we present some applications of the $A$-numerical radius parallelism and the $A$-numerical radius orthogonality. In particular, we will prove some inequalities for the $A$-numerical radius of semi-Hilbertian space operators. In order to achieve the goals of this section, we need some results.

In all what follows, we consider the Hilbert space $\mathbb{H}=\oplus_{i=1}^d\mathcal{H}$ equipped with the following inner-product:
$$\langle x, y\rangle=\sum_{k=1}^d\langle x_k\mid y_k\rangle,$$
 for all $x=(x_1,\cdots,x_d)\in \mathbb{H}$ and $y=(y_1,\cdots,y_d)\in \mathbb{H}$. Let $\mathbb{A}$ be a ${d\times d}$ operator diagonal matrix with diagonal entries are the positive operator $A$, i.e.
 \begin{equation*}
\mathbb{A}=\begin{pmatrix}A & 0 &\cdots& 0\\
0& A &\cdots& 0\\
\vdots & \vdots & \vdots & \vdots\\
0 & 0 &\cdots& A
\end{pmatrix}\in \mathcal{B}(\mathbb{H})^+.
\end{equation*}
 Then, $\mathbb{A}$ defines the following positive semi-definite sesquilinear form
\begin{align*}
\langle x,y\rangle_{\mathbb{A}}=\langle \mathbb{A}x, y\rangle=\sum_{k=1}^d\langle Ax_k\mid y_k\rangle=\sum_{k=1}^d\langle x_k\mid y_k\rangle_A,
\end{align*}
 for all $x=(x_1,\cdots,x_d),y=(y_1,\cdots,y_d)\in \mathbb{H}$.
We begin with the following lemma.
\begin{lemma}\label{lemmajdid01}
 Let $\mathbb{T}= (T_{ij})_{d \times d}$ be such that $T_{ij}\in \mathcal{B}_A(\mathcal{H})$ for all $i,j$. Then, $\mathbb{T}\in\mathcal{B}_{\mathbb{A}}(\mathbb{H})$. Moreover, we have
\begin{equation}\label{sharpmatrix}
\begin{pmatrix}
T_{11} & T_{12} &\cdots& T_{1d}\\
T_{21} & T_{22} &\cdots& T_{2d}\\
\vdots & \vdots & \vdots & \vdots\\
T_{d1} & T_{d2} &\cdots& T_{dd}\\
\end{pmatrix}^{\sharp_{\mathbb{A}}}=\begin{pmatrix}
T_{11}^{\sharp_A} &T_{21}^{\sharp_A}  &\cdots&T_{d1}^{\sharp_A} \\
 T_{12}^{\sharp_A} &   T_{22}^{\sharp_A} &\cdots& T_{d2}^{\sharp_A}\\
\vdots & \vdots & \vdots & \vdots\\
T_{1d}^{\sharp_A}  &   T_{2d}^{\sharp_A} &\cdots& T_{dd}^{\sharp_A}\\
\end{pmatrix}.
\end{equation}
\end{lemma}
\begin{proof}
By taking into account \eqref{badeh}, we need to show that there exists $\lambda > 0$ such that
$$\|\mathbb{A}\mathbb{T}x\|\leq \lambda \|\mathbb{A}x\|,$$
for all $x=(x_1,\cdots,x_d)\in \mathbb{H}$ or equivalently
$$\sum_{k=1}^d\left\|\sum_{j=1}^dAT_{kj}x_j\right\|^2\leq \lambda^2 \sum_{k=1}^d\left\|Ax_k\right\|^2.$$
Let $x=(x_1,\cdots,x_d)\in \mathbb{H}$. Since $T_{ij}\in \mathcal{B}_{A}(\mathcal{H})$ for all $i,j$, then by  \eqref{badeh} there exists $\mu_{ij} > 0$ such that
$$\|AT_{ij}x\|\leq \mu_{ij} \left\|Ax\right\|,$$
for all $x\in \mathcal{H}$ and $i,j\in\{1,\cdots, d\}$. So, we get
\begin{align*}
\|\mathbb{A}\mathbb{T}x\|^2
&= \sum_{k=1}^d\left\|\sum_{j=1}^dAT_{kj}x_j\right\|^2\\
&\leq\sum_{k=1}^d\left(\sum_{j=1}^d\|AT_{kj}x_j\|\right)^2\\
&\leq\sum_{k=1}^d\left(\sum_{j=1}^d\mu_{kj} \left\|Ax_j\right\|\right)^2\\
&\leq d\,(\max_{k,j}\{\mu^2_{kj}\})\left(\sum_{j=1}^d \left\|Ax_j\right\|\right)^2\\
&\leq d^2\,(\max_{k,j}\{\mu^2_{kj}\})\sum_{j=1}^d \left\|Ax_j\right\|^2=\lambda^2\|\mathbb{A}x\|^2,
\end{align*}
where $\lambda = {d}\,(\,\max_{k,j}\{\mu_{kj}\})$. Hence, $\mathbb{T}\in\mathcal{B}_{\mathbb{A}}(\mathbb{H})$. In order to prove \eqref{sharpmatrix}, let
$$\mathbb{S} =\begin{pmatrix}
T_{11}^{\sharp_A} &T_{21}^{\sharp_A}  &\cdots&T_{d1}^{\sharp_A} \\
 T_{12}^{\sharp_A} &   T_{22}^{\sharp_A} &\cdots& T_{d2}^{\sharp_A}\\
\vdots & \vdots & \vdots & \vdots\\
T_{1d}^{\sharp_A}  &   T_{2d}^{\sharp_A} &\cdots& T_{dd}^{\sharp_A}\\
\end{pmatrix}.$$
By using the fact that $T_{ij}\in \mathcal{B}_A(\mathcal{H})$ for all $i,j\in\{1,\cdots,d\}$, one obtains
\begin{align*}
\mathbb{A}\mathbb{S}
& =\begin{pmatrix}
AT_{11}^{\sharp_A} &AT_{21}^{\sharp_A}  &\cdots&AT_{d1}^{\sharp_A} \\
 AT_{12}^{\sharp_A} &   AT_{22}^{\sharp_A} &\cdots& AT_{d2}^{\sharp_A}\\
\vdots & \vdots & \vdots & \vdots\\
AT_{1d}^{\sharp_A}  &  A T_{2d}^{\sharp_A} &\cdots& AT_{dd}^{\sharp_A}\\
\end{pmatrix}\\
& =\begin{pmatrix}
T_{11}^{*}A &T_{21}^{*}A  &\cdots&T_{d1}^{*}A \\
T_{12}^{*}A &  T_{22}^{*}A &\cdots& T_{d2}^{*}A\\
\vdots & \vdots & \vdots & \vdots\\
T_{1d}^{*} A &  T_{2d}^{*}A &\cdots& T_{dd}^{*}A\\
\end{pmatrix}\\
& =\begin{pmatrix}
T_{11} & T_{12} &\cdots& T_{1d}\\
T_{21} & T_{22} &\cdots& T_{2d}\\
\vdots & \vdots & \vdots & \vdots\\
T_{d1} & T_{d2} &\cdots& T_{dd}\\
\end{pmatrix}^{*}\mathbb{A}.
\end{align*}
Finally, in order to get \eqref{sharpmatrix}, we shall need to show that $\mathcal{R}(\mathbb{S})\subseteq \overline{\mathcal{R}(\mathbb{A})}$.

Let $x=(x_1,x_2,\cdots,x_d)\in \mathbb{H}$ be arbitrary. Then
$$\mathbb{S}x=\begin{pmatrix}
\sum^d_{j=1}T^{\sharp_A}_{j1}x_j\\
\sum^d_{j=1}T^{\sharp_A}_{j2}x_j\\
\vdots\\
\sum^d_{j=1}T^{\sharp_A}_{jd}x_j\\
 \end{pmatrix}.$$
Since  $T^{\sharp_A}_{jk}x_j\in \overline{\mathcal{R}(A)},$ so $\sum^d_{j=1}T^{\sharp_A}_{jk}x_j\in \overline{\mathcal{R}(A)}$ for each $k=1,2,\ldots,d$. This implies that $\mathbb{S}x\in \overline{\mathcal{R}(A)} \oplus \overline{\mathcal{R}(A)} \ldots \oplus \overline{\mathcal{R}(A)}= \overline{\mathcal{R}(\mathbb{A})}.$ This completes the proof of the theorem.
\end{proof}

\begin{lemma}\label{tilde2020}
 Let $\mathbb{T}= (T_{ij})_{d \times d}$ be such that $T_{ij}\in \mathcal{B}_{A^{1/2}}(\mathcal{H})$ for all $i,j$. Then, $\mathbb{T}\in\mathcal{B}_{\mathbb{A}^{1/2}}(\mathbb{H})$. Moreover, we have
$$\widetilde{\mathbb{T}}=(\widetilde{T_{ij}})_{d \times d}.$$
\end{lemma}
\begin{proof}
By proceeding as in the proof of Lemma \ref{lemmajdid01}, we can show that there exists $\lambda > 0$ such that
$$\|\mathbb{T}x\|_\mathbb{A} \leq \lambda \|x\|_\mathbb{A},$$
for all $x=(x_1,\cdots,x_d)\in \mathbb{H}$. Hence, by \eqref{abbbbbbbb} we deduce that $\mathbb{T}\in\mathcal{B}_{\mathbb{A}^{1/2}}(\mathbb{H})$. So, by Proposition \ref{prop_arias} there exists a unique $\widetilde{\mathbb{T}}\in \mathcal{B}(\mathbf{R}(\mathbb{A}^{1/2}))$ such that $Z_{\mathbb{A}}\mathbb{T} =\widetilde{\mathbb{T}}Z_{\mathbb{A}}$. On the other hand, for all $x=(x_1,\cdots,x_d)\in \mathbb{H}$ we have
\begin{align*}
Z_{\mathbb{A}}\mathbb{T}x
&=\begin{pmatrix}A & 0 &\cdots& 0\\
0& A &\cdots& 0\\
\vdots & \vdots & \vdots & \vdots\\
0 & 0 &\cdots& A
\end{pmatrix}
\begin{pmatrix}
\sum_{j=1}^dT_{1j}x_j \\
\sum_{j=1}^dT_{2j}x_j \\
\vdots \\
\sum_{j=1}^dT_{dj}x_j
\end{pmatrix}  =\begin{pmatrix}
\sum_{j=1}^dAT_{1j}x_j \\
\sum_{j=1}^dAT_{2j}x_j \\
\vdots \\
\sum_{j=1}^dAT_{dj}x_j
\end{pmatrix}.\\
\end{align*}
Since, $T_{ij}\in \mathcal{B}_{A^{1/2}}(\mathcal{H})$ for all $i\in \{1,\cdots,d\}$. Then, by Proposition \ref{prop_arias} there exists $\widetilde{T_{ij}}\in \mathcal{B}(\mathbf{R}(A^{1/2}))$ such that $AT_{ij}=\widetilde{T_{ij}}A$ for all $i\in \{1,\cdots,d\}$. So, we get
\begin{align*}
Z_{\mathbb{A}}\mathbb{T}x
&=\begin{pmatrix}
\sum_{j=1}^d\widetilde{T_{1j}}Ax_j \\
\sum_{j=1}^d\widetilde{T_{2j}}Ax_j \\
\vdots \\
\sum_{j=1}^d\widetilde{T_{dj}}Ax_j
\end{pmatrix}\\
&=\begin{pmatrix}
\widetilde{T_{11}} & \widetilde{T_{12}} &\cdots& \widetilde{T_{1d}}\\
\widetilde{T_{21}} & \widetilde{T_{22}} &\cdots& \widetilde{T_{2d}}\\
\vdots & \vdots & \vdots & \vdots\\
\widetilde{T_{d1}} & \widetilde{T_{d2}} &\cdots& \widetilde{T_{dd}}
\end{pmatrix}
\begin{pmatrix}A & 0 &\cdots& 0\\
0& A &\cdots& 0\\
\vdots & \vdots & \vdots & \vdots\\
0 & 0 &\cdots& A
\end{pmatrix}
\begin{pmatrix}
x_1 \\
x_2 \\
\vdots \\
x_d
\end{pmatrix}\\
&=\begin{pmatrix}
\widetilde{T_{11}} & \widetilde{T_{12}} &\cdots& \widetilde{T_{1d}}\\
\widetilde{T_{21}} & \widetilde{T_{22}} &\cdots& \widetilde{T_{2d}}\\
\vdots & \vdots & \vdots & \vdots\\
\widetilde{T_{d1}} & \widetilde{T_{d2}} &\cdots& \widetilde{T_{dd}}
\end{pmatrix}
Z_{\mathbb{A}}x.
\end{align*}
Hence, we infer that $\widetilde{\mathbb{T}}=(\widetilde{T_{ij}})_{d \times d}.$
\end{proof}

The following refinement of the triangle inequality has been proved by F. Kittaneh et al. in \cite{OK2}.

%============================================================================================================

\begin{theorem}
\label{triangle inequality} Let $T,S\in \mathcal{B}(\mathcal{H})$. Then,
\begin{equation*}
\Vert T+S\Vert \leq 2\omega\left(\begin{array}{cc}
0 & T\\
S^{\ast } & 0
\end{array}\right) \leq \Vert T\Vert +\Vert S\Vert .
\end{equation*}
\end{theorem}
Now, we extend the above theorem as follows.
\begin{theorem}\label{th-p1}
Let $T_{12},T_{21}\in \mathcal{B}_{A}(\mathcal{H})$  and $\mathbb{A}=\left(\begin{array}{cc}
A&0 \\
0&A
\end{array}\right)\in \mathcal{B}(\mathcal{H}\oplus \mathcal{H})^+$. Then
\begin{equation}\label{pintu01}
\frac{1}{2}\|T_{12}+T_{21}\|_A\leq \omega_{\mathbb{A}}\left(\begin{array}{cc}
O&T_{12} \\
T^{\sharp_A}_{21}&O
\end{array}\right) \leq  \frac{1}{2}(\|T_{12}\|_A+\|T_{21}\|_A).
\end{equation}
\end{theorem}

\begin{proof}
Let $\mathbb{T}=\left(\begin{array}{cc}
O&T_{12} \\
T^{\sharp_A}_{21}&O
\end{array}\right)$. In view of Lemma \ref{tilde2020}, we have $\mathbb{T}\in \mathcal{B}_{\mathbb{A}^{1/2}}(\mathcal{H}\oplus \mathcal{H})$ as $\mathcal{B}_{A}(\mathcal{H}) \subseteq \mathcal{B}_{A^{1/2}}(\mathcal{H})$. So, by Proposition \ref{prop_arias} there exists $\widetilde{\mathbb{T}}\in \mathcal{B}(\mathbf{R}(\mathbb{A}^{1/2}))$ such that $Z_{\mathbb{A}}\mathbb{T}=\widetilde{\mathbb{T}}Z_{\mathbb{A}}$. Moreover, by applying Lemma \ref{tilde2020} together with Lemma \ref{lau}, we get
$$\widetilde{\mathbb{T}}=\left(\begin{array}{cc}
O&\widetilde{T_{12}} \\
\widetilde{T^{\sharp_A}_{21}}&O
\end{array}\right)=\left(\begin{array}{cc}
O&\widetilde{T_{12}} \\
\widetilde{T_{21}}^{*}&O
\end{array}\right)$$
So, by applying Lemma \ref{preprintkais} together with Theorem \ref{triangle inequality} we infer that
\[\frac{1}{2}\|\widetilde{T_{12}}+\widetilde{T_{21}}\|_{\mathcal{B}(\mathbf{R}(A^{1/2}))}\leq \omega(\widetilde{\mathbb{T}}) \leq  \frac{1}{2}(\|\widetilde{T_{12}}\|_{\mathcal{B}(\mathbf{R}(A^{1/2}))}+\|\widetilde{T_{21}}\|_{\mathcal{B}(\mathbf{R}(A^{1/2}))}).\]
On the other hand, it is not difficult to see that $\widetilde{T_{12}}+\widetilde{T_{21}}=\widetilde{T_{12}+T_{21}}$. Therefore, \eqref{pintu01} is proved by using Lemma \ref{preprintkais}.
\end{proof}

Using  the inequality \eqref{pintu01} we prove the following theorem.

\begin{theorem}
Let $T_{12},T_{21}\in \mathcal{B}_A(\mathcal{H})$ be such that $T_{12} \parallel_A T_{21}$. Let also $\mathbb{A}=\left(\begin{array}{cc}
A&0 \\
0&A
\end{array}\right)$. Then
\[ \omega_{\mathbb{A}}\left(\begin{array}{cc}
O&T_{12} \\
e^{-2i\beta}T^{\sharp}_{21}&O
\end{array}\right)= \tfrac{1}{2}(\|T_{12}\|_A+\|T_{21}\|_A),\]
for some real $\beta$.
\end{theorem}

\begin{proof}
We have $\|T_{12}+ e^{2i\beta}T_{21}\|_A=\|T_{12}\|_A+\|T_{21}\|_A$ for some  real $\beta$, as  $T_{12} \|_A T_{21}$. Therefore, using Theorem \ref{th-p1} we have
\begin{align*}
\|T_{12}\|_A+\|T_{21}\|_A
&= \|T_{12}+ e^{2i\beta}T_{21}\|_A\\
&\leq  2 \omega_{\mathbb{A}}\left(\begin{array}{cc}
O&T_{12} \\
e^{-2i\beta}T^{\sharp}_{21}&O
\end{array}\right)\\
&\leq \|T_{12}\|_A+\|T_{21}\|_A.
\end{align*}
It follows that
\[ \omega_{\mathbb{A}}\left(\begin{array}{cc}
O&T_{12} \\
e^{-2i\beta}T^{\sharp}_{21}&O
\end{array}\right) = \tfrac{1}{2}(\|T_{12}\|_A+\|T_{21}\|_A),\]
for some real $\beta$.
\end{proof}

Using the notion of $\mathbb{A}$-numerical radius orthogonality we obtain $\mathbb{A}$-numerical radius inequality for $d \times d$ operator matrices in the following theorem.

\begin{theorem}\label{th-1}
Let $\mathbb{T}=(T_{ij})$ be a ${d \times d}$ operator matrix where $T_{ij}\in \mathcal{B}_{A^{1/2}}(\mathcal{H})$. Then,
\[\omega_\mathbb{A}(\mathbb{T})\geq \max \left \{ \omega_A(T_{ii}), \omega_{\mathbb{A}}(\mathbb{S}_{i}): 1\leq i\leq d\right \}\]
where for each $i\in\{1,\cdots,d\}$, the operator matrix $\mathbb{S}_i=(s^{i}_{jk})_{d \times d}$ is defined as $s^{i}_{jk}=O$ if $j=i$ or $k=i$ and $s^{i}_{jk}=T_{jk}$ otherwise, i.e.
 	\[
 	\mathbb{S}_i=\begin{pmatrix}
 	T_{11} &\ldots&T_{1(i-1)}&O  &T_{1(i+1)}&\ldots&T_{1d}\\
 	\vdots &\vdots &\ldots&\vdots &\vdots&  \vdots   & \ldots \\
 	T_{(i-1)1} &\ldots&T_{(i-1)(i-1)}&O  &T_{(i-1)(i+1)}&\ldots&T_{(i-1)d}\\
 	O &\ldots&O&O&O&\ldots&O \\
 	T_{(i+1)1} &\ldots&T_{(i+1)(i-1)}&O  &T_{(i+1)(i+1)}&\ldots&T_{(i+1)d}\\
 	\vdots &\ldots&\vdots &\vdots&  \vdots   & \ldots &\vdots\\
 	T_{d1} &\ldots&T_{d(i-1)}&O  &T_{d(i+1)}&\ldots&T_{dd}	
 	\end{pmatrix}.
 	\]
\end{theorem}

\begin{proof}
First we prove that $\omega_\mathbb{A}(\mathbb{T})\geq \omega_A(T_{ii})$ for all $i\in\{1,\cdots,d\}$. By definition of the $A$-numerical radius of the operator $T_{11}$, there exists a sequence of $A$-unit vectors $\{x_n\}$ in $\mathcal{H}$ such that
\[\displaystyle\lim_{n\rightarrow +\infty}|\langle T_{11}x_n\mid x_n\rangle_A|=\omega_A(T_{11}).\]
Let $X_n=(x_n,0,\cdots,0)\in \mathbb{H}$. Then $\|X_n\|_\mathbb{A}=\|x_n\|_A=1.$ Therefore, from an easy calculation we have
\begin{align*}
|\langle \mathbb{T}X_n,X_n\rangle_\mathbb{A}| &= |\langle T_{11}x_n\mid x_n\rangle_A|\\
\Rightarrow \omega_\mathbb{A}(\mathbb{T}) &\geq |\langle T_{11}x_n\mid x_n\rangle_A|\\
\Rightarrow \omega_\mathbb{A}(\mathbb{T}) &\geq \displaystyle\lim_{n\rightarrow +\infty}|\langle T_{11}x_n\mid x_n\rangle_A|\\
\Rightarrow \omega_\mathbb{A}(\mathbb{T}) &\geq \omega_A(T_{11}).
\end{align*}
Similarly, we can show that $\omega_\mathbb{A}(\mathbb{T})\geq \omega_A(T_{ii})$ for all $i=2,3,\ldots,d$.

Next we will prove that $\omega_{\mathbb{A}}(\mathbb{T})\geq \omega_\mathbb{A}(\mathbb{S}_{i})$ for all $i\in\{1,\cdots,d\}$. Let us assume that $\mathbb{M}_i=\mathbb{T}-\mathbb{S}_i$. We first show that $\mathbb{S}_i \perp_{\omega_\mathbb{A}} \mathbb{M}_i$.
%{\bf Let $X=(x_1,\cdots,x_d)\in \mathbb{H}$ be such that $\|X\|_\mathbb{A}=1$, i.e., $\|x_1\|_A^2+\|x_2\|_A^2+\ldots+\|x_d\|_A^2=1.$  Then
%\[\left|\langle \mathbb{S}_{i}X,X\rangle_{\mathbb{A}}\right|=\left|\sum^d_{\substack{j=1,\\j\neq i}}\sum^d_{\substack{k=1,\\k\neq i}}\langle T_{jk}x_k\mid x_j\rangle_A \right |.\]
%I don't see the importance of this }
Let $\{X_m\}=\{(x_{m1},x_{m2},\cdots,x_{md})\}$ be a sequence of $\mathbb{A}$-unit vector in $\mathbb{H}$, i.e. $\|X_m\|_\mathbb{A}=1$, i.e. $\|x_{m1}\|_A^2+\|x_{m2}\|_A^2+\cdots+\|x_{md}\|_A^2=1$ such that
\[\displaystyle\lim_{m\rightarrow +\infty}|\langle \mathbb{S}_{i}X_m,X_m\rangle_{\mathbb{A}}|=\omega_{\mathbb{A}}(\mathbb{S}_{i}).\]
We claim that there exist sequence $\{Z_m\}$ in $\mathbb{H}$ of the form
$\{(z_{m1},z_{m2},\ldots,z_{m (i-1)},0,$ $z_{m (i+1)},\ldots, z_{md})\}$
such that $\|Z_m\|_\mathbb{A}=1$ and
\[\displaystyle\lim_{m\rightarrow +\infty}|\langle \mathbb{S}_{i}Z_m,Z_m\rangle_{\mathbb{A}}|=\omega_{\mathbb{A}}(\mathbb{S}_{i}).\]
Suppose that $\sum^d_{j=1,j\neq i}\|x_{mj}\|_A^2<1$. Let $\alpha=\frac{1}{\sqrt{\sum^d_{j=1,j\neq i}\|x_{mj}\|_A^2}}$ and
$$Z_m=\alpha(x_{m1},x_{m2},\ldots,x_{m (i-1)},0,x_{m (i+1)},\ldots, x_{md}).$$
 Then $\|Z_m\|_\mathbb{A}=1$. Since $\alpha > 1$, we have
\begin{eqnarray*}
|\langle \mathbb{S}_{i}Z_m,Z_m\rangle_{\mathbb{A}}|&=&\alpha^2 \left|\sum^n_{j=1,j\neq i}\sum^n_{k=1,k\neq i}\langle T_{jk}x_{mk} \mid x_{mj}\rangle_A \right |\\
&=& \alpha^2\left|\langle \mathbb{S}_{i}X_m,X_m\rangle_{\mathbb{A}}\right|\\
&>& \left|\langle \mathbb{S}_{i}X_m,X_m\rangle_{\mathbb{A}}\right|.
\end{eqnarray*}
So, we infer that
$$\omega_{\mathbb{A}}(\mathbb{S}_i)\geq \displaystyle\lim_{m\rightarrow +\infty}|\langle \mathbb{S}_{i}Z_m,Z_m\rangle_{\mathbb{A}}|\geq \displaystyle\lim_{m\rightarrow +\infty}|\langle \mathbb{S}_{i}X_m,X_m\rangle_{\mathbb{A}}|=\omega_{\mathbb{A}}(\mathbb{S}_i).$$
This implies that
\[\displaystyle\lim_{m\rightarrow +\infty}|\langle \mathbb{S}_{i}Z_m,Z_m\rangle_{\mathbb{A}}|=\omega_{\mathbb{A}}(\mathbb{S}_{i}).\]
This proves our claim.

Now, an easy calculation shows that $\langle \mathbb{M}_{i}Z_m,Z_m\rangle_{\mathbb{A}}=0.$ Therefore, for each $\beta \in [0, 2\pi)$,
\[\Re e\left ( e^{i\beta}  \langle Z_m,\mathbb{S}_{i}Z_m\rangle_{\mathbb{A}}  \langle \mathbb{M}_{i}Z_m,Z_m\rangle_{\mathbb{A}} \right )=0.\]
Therefore from Theorem \ref{mainortho} we have  $\mathbb{S}_{i} \perp_{\omega_\mathbb{A}} \mathbb{M}_{i}.$ So,
  $$\omega_{\mathbb{A}}(\mathbb{T})= \omega_{\mathbb{A}}(\mathbb{S}_i+1\times \mathbb{M}_i)\geq \omega_{\mathbb{A}}(\mathbb{S}_i) ~~\textit{for all}~~ i=1,2,\ldots,d. $$
Hence we conclude that
   \[\omega_{\mathbb{A}}(\mathbb{T})\geq \max \left \{ \omega_A(T_{ii}), \omega_{\mathbb{A}}(\mathbb{S}_{i}): 1\leq i\leq d\right \},\]
This completes the proof of the theorem.
\end{proof}

Based on Theorem \ref{th-1} we obtain the following inequality.

\begin{corollary}\label{cor-1}
Let $\mathbb{T}=(T_{ij})$ be a ${d\times d}$ operator matrix where $T_{ij}\in \mathcal{B}_{A^{1/2}}(\mathcal{H})$. Then,
\[\omega_{\mathbb{A}}(\mathbb{T})\geq \max \left \{ \omega_A(T_{kk}), \omega_{\mathbb{A}'}\left(\begin{array}{cc}
T_{ii}&T_{ij} \\
T_{ji}&T_{jj}
\end{array}\right): 1\leq k\leq d,1\leq i<j\leq d\right \}\]
where $\mathbb{A}'=\begin{pmatrix}A&0\\0&A\end{pmatrix}$.
\end{corollary}

\begin{proof}
Let $\mathbb{D}=\textit{diag}(A,A,\ldots,A)$ be a ${(d-1) \times (d-1)}$ operator matrix and for $i\in\{1,\cdots,d\}$ we set the following ${(d-1) \times (d-1)}$ operator matrix
 	\[
\mathbb{R}_i=\begin{pmatrix}
T_{11} &\ldots&T_{1(i-1)} &A_{1(i+1)}&\ldots&T_{1d}\\
 \vdots &\ldots &\vdots &\vdots&  \ldots   & \vdots \\
T_{(i-1)1} &\ldots&T_{(i-1)(i-1)} &T_{(i-1)(i+1)}&\ldots&T_{(i-1)d}\\
T_{(i+1)1} &\ldots&T_{(i+1)(i-1)}  &T_{(i+1)(i+1)}&\ldots&T_{(i+1)d}\\
 \vdots &\ldots&\vdots &  \vdots   & \ldots &\vdots\\
T_{d1} &\ldots&T_{d(i-1)} &T_{d(i+1)}&\ldots&T_{dd}	
 \end{pmatrix}.
 \]
It can be seen that $\omega_{\mathbb{A}}(\mathbb{S}_i)=\omega_\mathbb{D}(\mathbb{R}_i)$ where $\mathbb{S}_i$ is defined as in Theorem \ref{th-1}. Moreover, by applying  Theorem \ref{th-1} repeatedly on $\mathbb{R}_i$ for each $i=1,2,\ldots,d$ we get our desired inequality.
\end{proof}

In the following lemma, we will show that A-numerical radius of semi-Hilbertican space operators satisfying weak A-unitary invariance property.

\begin{lemma}
Let $T\in \mathcal{B}_{A^{1/2}}(\mathcal{H})$. Then,
\begin{equation}\label{aunitary}
\omega_{A}(U^{\sharp}TU)=\omega_{A}(T),
\end{equation}
for every $A$-unitary operator $U\in \mathcal{B}_{A}(\mathcal{H})$.
\end{lemma}
\begin{proof}
Since $T\in \mathcal{B}_{A^{1/2}}(\mathcal{H})$ and $U\in \mathcal{B}_{A}(\mathcal{H})\subseteq\mathcal{B}_{A^{1/2}}(\mathcal{H})$, then by Proposition \ref{prop_arias} there exists a unique $\widetilde{T},\widetilde{U}\in \mathcal{B}(\mathbf{R}(A^{1/2}))$ such that $Z_AT =\widetilde{T}Z_A$ and $Z_AU=\widetilde{U}Z_A$. On the other hand, since $U\in \mathcal{B}_A(\mathcal{H})$ is an $A$-unitary operator, then
$$(U^{\sharp_A})^{\sharp_A} U^{\sharp_A}=U^{\sharp_A} U=P_A.$$
This gives
$$\widetilde{U^{\sharp_A} U}=\widetilde{(U^{\sharp_A})^{\sharp_A} U^{\sharp_A}}=\widetilde{P_A}.$$
So, it can be seen by using Lemma \ref{lau} that
$$\widetilde{U}^* \widetilde{U}=\widetilde{U} \widetilde{U}^*=I_{\mathbf{R}(A^{1/2})}.$$
Hence, $\widetilde{U}$ is an unitary operator on the Hilbert space $\mathbf{R}(A^{1/2})$. So, we get
$$\omega(\widetilde{U}^{*}\widetilde{T}\widetilde{U})=\omega(\widetilde{T}).$$
This implies, by Lemma \ref{lau} that
$$\omega(\widetilde{U^{\sharp}}\widetilde{T}\widetilde{U})=\omega(\widetilde{T}),$$
which in turn yields that
$$\omega(\widetilde{U^{\sharp}TU})=\omega(\widetilde{T}),$$
Therefore, we obtain \eqref{aunitary} by using Lemma \ref{preprintkais}.
\end{proof}

Next, we obtain a lower bound for $\mathbb{A}$-numerical radius of $2 \times 2$ operator matrices. To obtain this we need the following lemma.

\begin{lemma}\label{lem-1}
Let $\mathbb{T}=\left(\begin{array}{cc}
T_{11}&T_{12} \\
T_{21}&T_{22}
\end{array}\right)$ where $T_{ij}\in \mathcal{B}_{A^{1/2}}(\mathcal{H})$ and $\mathbb{A}=\left(\begin{array}{cc}
A&0\\
0&A
\end{array}\right)$. Then
\[\omega_\mathbb{A}(\mathbb{T})=\omega_\mathbb{A}\left(\begin{array}{cc}
T_{11}&iT_{12} \\
-iT_{21}&T_{22}
\end{array}\right).\]
\end{lemma}
\begin{proof}
Let $\mathbb{U}=\left(\begin{array}{cc}
-iI&O \\
O&I
\end{array}\right).$ By using Lemma \ref{lemmajdid01}, one gets
$$\mathbb{U}^{\sharp_{\mathbb{A}}}\mathbb{U}=\left(\begin{array}{cc}
iP_A&O \\
O&P_A
\end{array}\right)\left(\begin{array}{cc}
-iI&O \\
O&I
\end{array}\right)=\left(\begin{array}{cc}
P_A&O \\
O&P_A
\end{array}\right)=P_{\mathbb{A}},$$
where $P_{\mathbb{A}}$ is denoted to be the orthogonal projection onto $\overline{\mathcal{R}(\mathbb{A})}$. Similarly, we show that $(\mathbb{U}^{\sharp_\mathbb{A}})^{\sharp_\mathbb{A}} \mathbb{U}^{\sharp_\mathbb{A}}=P_\mathbb{A}.$ Hence, $\mathbb{U}$ is an $\mathbb{A}$-unitary operator. So, the desired equality follows immediately by using \eqref{aunitary} and remarking that $\omega_\mathbb{A}(\mathbb{T})=\omega_\mathbb{A}(P_{\mathbb{A}}\mathbb{T})$.
\end{proof}

Now we are in a position to prove the following lower bound for $\mathbb{A}$-numerical radius of $2 \times 2$ operator matrices where $\mathbb{A}= \textit{diag}(A,A)$.

\begin{theorem}\label{th-2}
Let $\mathbb{T}=\left(\begin{array}{cc}
T_{11}&T_{12} \\
T_{21}&T_{22}
\end{array}\right)$ where $T_{ij}\in \mathcal{B}_{A^{1/2}}(\mathcal{H})$ and $\mathbb{A}=\left(\begin{array}{cc}
A&0\\
0&A
\end{array}\right)$. Then
\[\omega_{\mathbb{A}}(\mathbb{T})\geq \max\left \{\omega_A(T_{11}),\omega_A(T_{22}), \alpha, \beta \right\}\]
where
$$\alpha=\sqrt{m^2_A\left (\frac{T_{11}+T_{22}}{2}\right)+\omega^2_A\left(\frac{T_{12}+T_{21}}{2}\right)},$$
$$\beta=\sqrt{m^2_A\left(\frac{T_{11}+T_{22}}{2}\right)+\omega^2_A\left(\frac{T_{12}-T_{21}}{2}\right)}.$$
\end{theorem}

\begin{proof}
First we prove that $\omega_{\mathbb{A}}(T)\geq \omega_A(T_{ii})$ for all $i=1,2.$ Let $\{x_n\}$ be a sequence of $A$-unit vectors in $\mathcal{H}$ such that
\[\displaystyle\lim_{n\rightarrow +\infty}|\langle T_{11}x_n\mid x_n\rangle_A|=\omega_A(T_{11}).\]
Let $X_n=(x_n,0)\in \mathcal{H}\oplus \mathcal{H}$. Then $\|X_n\|_\mathbb{A}=\|x_n\|_A=1.$ Therefore, from an easy calculation we have
\begin{align*}
|\langle T_{11}X_n,X_n\rangle_{\mathbb{A}}|
 &= |\langle T_{11}x_n\mid x_n\rangle_A|\\
\Rightarrow \omega_{\mathbb{A}}(\mathbb{T}) &\geq |\langle T_{11}x_n\mid x_n\rangle_A|\\
\Rightarrow \omega_{\mathbb{A}}(\mathbb{T}) &\geq \displaystyle\lim_{n\rightarrow +\infty}|\langle T_{11}x_n\mid x_n\rangle_A|\\
\Rightarrow \omega_{\mathbb{A}}(\mathbb{T}) &\geq \omega_A(T_{11}).
\end{align*}
Similarly, we can show that $\omega_{\mathbb{A}}(\mathbb{T})\geq \omega_A(T_{22}).$

Now we show that $\omega_{\mathbb{A}}(\mathbb{T})\geq \alpha$ and $\omega_{\mathbb{A}}(\mathbb{T})\geq \beta$. We consider $\mathbb{T}=\mathbb{P}+\mathbb{S}$ where $\mathbb{P}=\left(\begin{array}{cc}
T_{11}&O\\
O&T_{22}
\end{array}\right)$ and $\mathbb{S}=\left(\begin{array}{cc}
O&T_{12} \\
T_{21}&O
\end{array}\right).$
Let $\{z_n\}$ be a sequence of $A$-unit vectors in $\mathcal{H}$ such that
\[\displaystyle\lim_{n\rightarrow +\infty}\left|\langle (\tfrac{T_{12}+T_{21}}{2})z_n\mid z_n\rangle_A\right|=\omega_A\left(\tfrac{T_{12}+T_{21}}{2}\right).\]
Let $Z_{1n}=\frac{1}{\sqrt{2}}(z_n,z_n)$ and  $Z_{2n}=\frac{1}{\sqrt{2}}(-z_n,z_n)$ be in $\mathcal{H}\oplus \mathcal{H}$. Then by an easy calculation we see that
\[\langle \mathbb{S}Z_{1n},Z_{1n}\rangle_{\mathbb{A}}=-\langle \mathbb{S}Z_{2n},Z_{2n}\rangle_{\mathbb{A}}=\langle \tfrac{T_{12}+T_{21}}{2}z_n\mid z_n\rangle_A,\]
\[\langle \mathbb{P}Z_{1n},Z_{1n}\rangle_{\mathbb{A}}=\langle \mathbb{P}Z_{2n},Z_{2n}\rangle_{\mathbb{A}}=\langle \tfrac{T_{11}+T_{22}}{2}z_n\mid z_n\rangle_A.\]
From this we observe that either one of the following holds:
$$(i)~~ \Re e\left \{  \langle \mathbb{P}Z_{1n},Z_{1n}\rangle_{\mathbb{A}}  \overline {\langle \mathbb{S}Z_{1n},Z_{1n}\rangle_{\mathbb{A}} } \right \}\geq 0.$$
$$(ii)~~ \Re e\left \{  \langle \mathbb{P}Z_{2n},Z_{2n}\rangle_{\mathbb{A}}  \overline {\langle \mathbb{S}Z_{2n},Z_{2n}\rangle_{\mathbb{A}} } \right \}\geq 0.$$
Without loss of generality we assume that $(i)$ holds. Then, we have
\begin{align*}
\omega^2_\mathbb{A}(\mathbb{T})
&=\omega^2_\mathbb{A}(\mathbb{P}+\mathbb{S})\\
&\geq |\langle(\mathbb{P}+\mathbb{S})Z_{1n},Z_{1n}\rangle_{\mathbb{A}}|^2\\
&= |\langle \mathbb{P}Z_{1n},Z_{1n}\rangle_{\mathbb{A}}+\langle \mathbb{S}Z_{1n},Z_{1n}\rangle_{\mathbb{A}}|^2\\
&= |\langle \mathbb{P}Z_{1n},Z_{1n}\rangle_{\mathbb{A}}|^2+|\langle \mathbb{S}Z_{1n},Z_{1n}\rangle_{\mathbb{A}}|^2+2\Re e\left \{  \langle \mathbb{P}Z_{1n},Z_{1n}\rangle_{\mathbb{A}}  \overline {\langle \mathbb{S}Z_{1n},Z_{1n}\rangle_{\mathbb{A}} } \right \}\\
&\geq |\langle \mathbb{P}Z_{1n},Z_{1n}\rangle_{\mathbb{A}}|^2+|\langle \mathbb{S}Z_{1n},Z_{1n}\rangle_{\mathbb{A}}|^2\\
&= \left|\langle \left(\tfrac{T_{11}+T_{22}}{2}\right)z_n\mid z_n\rangle_A\right|^2+\left|\langle \left(\tfrac{T_{12}+T_{21}}{2}\right)z_n\mid z_n\rangle_A\right|^2\\
&\geq m^2_A\left (\tfrac{T_{11}+T_{22}}{2}\right)+\left|\langle \left(\tfrac{T_{12}+T_{21}}{2}\right)z_n\mid z_n\rangle_A\right|^2\\
&\geq m^2_A\left (\tfrac{T_{11}+T_{22}}{2}\right)+\displaystyle\lim_{n\rightarrow +\infty}\left|\langle \left(\tfrac{T_{12}+T_{21}}{2}\right)z_n\mid z_n\rangle_A\right|^2\\
&= m^2_A\left (\tfrac{T_{11}+T_{22}}{2}\right)+\omega^2_A\left(\tfrac{T_{12}+T_{21}}{2}\right)\\
\Rightarrow \omega_{\mathbb{A}}(\mathbb{T}) &\geq \sqrt{m^2_A\left (\tfrac{T_{11}+T_{22}}{2}\right)+\omega^2_A\left(\tfrac{T_{12}+T_{21}}{2}\right)}.
\end{align*}
To show $\omega_{\mathbb{A}}(\mathbb{T})\geq \beta$, we consider the operator matrix $\left(\begin{array}{cc}
T_{11}&iT_{12} \\
-iT_{21}&T_{22}
\end{array}\right).$ By replacing $T_{12},T_{21}$ by $iT_{12},-iT_{21}$ respectively in the above last inequality, and by using Lemma \ref{lem-1} we have
\begin{align*}
\omega_{\mathbb{A}}(\mathbb{T}) &\geq \sqrt{m^2_A\left (\tfrac{T_{11}+T_{22}}{2}\right)+\omega^2_A\left(\tfrac{iT_{12}-iT_{21}}{2}\right)}\\
&=\sqrt{m^2_A\left (\tfrac{T_{11}+T_{22}}{2}\right)+\omega^2_A\left(\tfrac{T_{12}-T_{21}}{2}\right)}.
\end{align*}
Therefore, we conclude that
\[\omega_{\mathbb{A}}(T)\geq \max\left \{\omega_A(T_{11}),\omega_A(T_{22}), \alpha, \beta \right\}.\]
Hence, completes the proof of the theorem.
\end{proof}

\begin{remark}
Here we would like to remark that the acquired inequality in Theorem \ref{th-2} generalized as well as improves on the inequality obtained in \cite[Th. 3.1]{malpaulsen}.
\end{remark}

The following corollary is an immediate consequence of Corollary \ref{cor-1} and Theorem \ref{th-2}.
\begin{corollary}\label{cor-2}
Let $\mathbb{T}=(T_{ij})$ be a ${d \times d}$ operator matrix where $T_{ij}\in \mathcal{B}_{A^{1/2}}(\mathcal{H})$. Then,
\[\omega_{\mathbb{A}}(\mathbb{T})\geq \max \left \{ \omega_A(T_{kk}), \alpha_{ij},\beta_{ij} : 1\leq k\leq d, 1\leq i<j \leq d\right \}\]
where
$$\alpha_{ij}=\sqrt{m^2_A\left (\tfrac{T_{ii}+T_{jj}}{2}\right)+\omega^2_A\left(\tfrac{T_{ij}+T_{ji}}{2}\right)},$$
$$\beta_{ij}=\sqrt{m^2_A\left(\tfrac{T_{ii}+T_{jj}}{2}\right)+\omega^2_A\left(\tfrac{T_{ij}-T_{ji}}{2}\right)}.$$
\end{corollary}

\begin{remark}
Here we remark that the inequality in Corollary \ref{cor-2} generalised and improves on the existing inequality in \cite[Th. 3.3]{malpaulsen}.
The inequality in \cite[Th. 3.3]{malpaulsen} follows from  Corollary \ref{cor-2} by considering $A=I.$
\end{remark}

Our next result reads as follows.

\begin{theorem}\label{th-3}
Let $\mathbb{T}=(T_{ij})$ be a ${d \times d}$ operator matrix where $T_{ij}\in \mathcal{B}_{A^{1/2}}(\mathcal{H})$ and $T_{ij}=O$ when $i>j$, i.e.
 	\[	\mathbb{T}= \begin{pmatrix}
	T_{11} & T_{12} &T_{13} &\ldots&T_{1d}  \\
	O & T_{22}   & T_{23} &\ldots  &T_{2d} \\
	\vdots  &\ddots&\ddots& \ddots&\vdots\\
	\vdots  & &\ddots& \ddots&T_{(n-1)d}\\
	O&\ldots&\ldots&O&T_{dd}
	\end{pmatrix}.\]
 Then, we have
\[\omega_{\mathbb{A}}(\mathbb{T})\geq \max \left \{ \omega_A(T_{kk}), \tfrac{\|T_{ij}\|_A}{2}: 1\leq k\leq d, 1\leq i<j\leq d\right \}.\]
\end{theorem}

\begin{proof}
We have from Corollary \ref{cor-1} that
\[\omega_{\mathbb{A}}(\mathbb{T})\geq \max \left \{ \omega_A(T_{kk}), \omega_{\mathbb{A}'}\left(\begin{array}{cc}
T_{ii}&T_{ij} \\
O&T_{jj}
\end{array}\right): 1\leq k\leq d,1\leq i<j\leq d\right \}\]
where $\mathbb{A}'=\begin{pmatrix}A&0\\0&A\end{pmatrix}$. So, to prove this theorem we only show that
$$\omega_{\mathbb{A}'}\left(\begin{array}{cc}
T_{ii}&T_{ij} \\
O&T_{jj}
\end{array}\right)\geq \frac{\|T_{ij}\|_A}{2}.$$
Let us assume that $\mathbb{P}=\left(\begin{array}{cc}
O&T_{ij} \\
O&O
\end{array}\right)$ and $\mathbb{S}=\left(\begin{array}{cc}
T_{ii}&O \\
O&T_{jj}
\end{array}\right)$. It is not difficult to verify that $\|\mathbb{P}\|_{\mathbb{A}'}=\|T_{ij}\|_A$. Moreover, since $\mathbb{A}'\mathbb{P}^2=O$, then by \cite[Cor. 2.2]{kais01} we have $\omega_{\mathbb{A}'}(\mathbb{P})=\frac{\|\mathbb{P}\|_{\mathbb{A}'}}{2}=\frac{\|T_{ij}\|_A}{2}$. We claim that $\mathbb{P}\perp_{\omega_{\mathbb{A}'}} \mathbb{S}.$ In view of \eqref{aseminorm}, there exists a sequence of $A$-unit vectors $\{x_m\}$ in $\mathcal{H}$ such that
\[\displaystyle\lim_{m \rightarrow +\infty}\|T_{ij}x_m\|_A=\|T_{ij}\|_A.\]
Let $Z_{1m}=\frac{1}{\sqrt{2}\|T_{ij}\|_A}(T_{ij}x_m,\|T_{ij}\|_Ax_m)$ and $Z_{2m}=\frac{1}{\sqrt{2}\|T_{ij}\|_A}(-T_{ij}x_m,\|T_{ij}\|_Ax_m)$ be in $\mathcal{H}\oplus \mathcal{H}$. Then it can be checked that, for the sequences of $\mathbb{A'}$-unit vectors $\{Z_{1m}\}$ and $\{Z_{2m}\},$
\begin{align*}
\displaystyle\lim_{m \rightarrow +\infty}\langle \mathbb{P} Z_{1m},Z_{1m}\rangle_{\mathbb{A}'}=\frac{\|T_{ij}\|_A}{2}, i.e. \displaystyle\lim_{m \rightarrow +\infty}|\langle \mathbb{P} Z_{1m},Z_{1m}\rangle_{\mathbb{A}'}|=\omega_{\mathbb{A}'}(\mathbb{P})\\
\displaystyle\lim_{m \rightarrow +\infty}\langle \mathbb{P} Z_{2m},Z_{2m}\rangle_{\mathbb{A}'}=\frac{-\|T_{ij}\|_A}{2}, i.e. \displaystyle\lim_{m \rightarrow +\infty}|\langle \mathbb{P} Z_{2m},Z_{2m}\rangle_{\mathbb{A}'}|=\omega_{\mathbb{A}'}(\mathbb{P})\\
\end{align*}
Also we have $\langle \mathbb{S}Z_{1m},Z_{1m}\rangle_{\mathbb{A}'}=\langle \mathbb{S} Z_{2m},Z_{2m}\rangle_{\mathbb{A}'}.$ Therefore for any $\beta \in [0,2\pi)$, either one of the following holds:
$$(i)~~ \Re e\left \{ e^{i\beta} \langle Z_{1m},\mathbb{P}Z_{1m}\rangle_{\mathbb{A}'}   {\langle \mathbb{S}Z_{1m},Z_{1m}\rangle_{\mathbb{A}'} } \right \}\geq 0.$$
$$(ii)~~ \Re e\left \{  e^{i\beta} \langle Z_{2m},\mathbb{P}Z_{2m}\rangle_{\mathbb{A}'}  {\langle \mathbb{S}Z_{2m},Z_{2m}\rangle_{\mathbb{A}'} } \right \}\geq 0.$$
Therefore from Theorem \ref{mainortho}, we have
$$\mathbb{P} \perp_{\omega_{\mathbb{A}'}} \mathbb{S}.$$
 So, $\omega_{\mathbb{A}'}\left(\begin{array}{cc}
T_{ii}&T_{ij} \\
O&T_{jj}
\end{array}\right)=\omega_{\mathbb{A}'}(\mathbb{P}+1\times \mathbb{S}) \geq \omega_{\mathbb{A}'}(\mathbb{P})=\frac{\|T_{ij}\|_A}{2}.$  Hence, completes the proof.

%{\bf I don't understand very well the proof: in order to prove that $\mathbb{P} \perp_{\omega_{\mathbb{A}'}} \mathbb{S}.$ we have two conditions Please see Theorem 2.5.}
\end{proof}

\begin{remark}
In particular, if we consider $A=I$ in Corollary \ref{cor-2} then we get the inequality in \cite[Th. 3.4]{malpaulsen}.
\end{remark}

%%%%%%%%%%%%%%%%%%%%%%%%%%%%%%%%%%%%%%%%%%%%%%%%%%%%%%%%%%%%%%%%%%%%%%%%%%%%%%%%%%%%%%%%%%%%%%%%%%%%%%%%
%-----------------------------------------------{chapter}{Bibliography}------------------------------------------

\end{document}